\documentclass[12pt, a4paper]{amsart}

\usepackage{amssymb, amsmath}
\usepackage[all]{xy}
\usepackage[inline]{enumitem}
\usepackage{hyphenat}
\usepackage[colorlinks,linkcolor=blue,citecolor=blue,urlcolor=teal]{hyperref}
\usepackage{xcolor}
\usepackage{mathtools}
\usepackage{tikz-cd}
\usepackage[labelformat=empty]{caption}
 \usepackage[ left=3cm, right=3cm, top = 2.8cm, bottom = 2.8cm ]{geometry}

\makeatletter
\def\@tocline#1#2#3#4#5#6#7{\relax
  \ifnum #1>\c@tocdepth 
  \else
    \par \addpenalty\@secpenalty\addvspace{#2}%
    \begingroup \hyphenpenalty\@M
    \@ifempty{#4}{%
      \@tempdima\csname r@tocindent\number#1\endcsname\relax
    }{%
      \@tempdima#4\relax
    }%
    \parindent\z@ \leftskip#3\relax \advance\leftskip\@tempdima\relax
    \rightskip\@pnumwidth plus4em \parfillskip-\@pnumwidth
    #5\leavevmode\hskip-\@tempdima
      \ifcase #1
      \or\or \hskip 2em \or \hskip 2em \else \hskip 3em \fi%
      #6\nobreak\relax
    \dotfill\hbox to\@pnumwidth{\@tocpagenum{#7}}\par
    \nobreak
    \endgroup
  \fi}
\makeatother

\newcommand{\eq}[2]{\begin{equation}\label{#1}#2 \end{equation}}
\newcommand{\ml}[2]{\begin{multline}\label{#1}#2 \end{multline}}
\newcommand{\mlnl}[1]{\begin{multline*}#1 \end{multline*}}

\newcommand{\arir}{\ar@{^{(}->}}
\newcommand{\aril}{\ar@{_{(}->}}
\newcommand{\are}{\ar@{>>}}

\newcommand{\xr}[1] {\xrightarrow{#1}}

\newtheorem{lem}{Lemma}[section]
\newtheorem{thm}[lem]{Theorem}

\newtheorem{prop}[lem]{Proposition}

\newtheorem{cor}[lem]{Corollary}

\theoremstyle{definition}
\newtheorem{defn}[lem]{Definition}
\newtheorem{defn-prop}[lem]{Definition-Proposition}

\newtheorem{para}[lem]{}

\theoremstyle{remark}

\newtheorem{rmk}[lem]{Remark}

\newtheorem{exs-rmks}[lem]{Examples and Remarks}

\newtheorem{claim}{Claim}[lem]
\newtheorem*{claim*}{Claim}

\newtheorem*{ack}{Acknowledgments}

\newcounter{zaehler} 
\numberwithin{equation}{lem}

\newcommand{\N}{\mathbb{N}}

\newcommand{\Z}{\mathbb{Z}}

\renewcommand{\P}{\mathbf{P}}

\newcommand{\F}{\mathbf{F}}

\newcommand{\sF}{\mathcal{F}}
\newcommand{\sG}{\mathcal{G}}
\newcommand{\sH}{\mathcal{H}}

\newcommand{\sO}{\mathcal{O}}

\newcommand{\sK}{\mathcal{K}}

\newcommand{\sV}{\mathcal{V}}

\newcommand{\fm}{\mathfrak{m}}

\newcommand{\fV}{\mathfrak{V}}

\newcommand{\tC}{{\widetilde{C}}}

\newcommand{\ul}[1]{{\underline{#1}}}

\newcommand{\Hom}{\operatorname{Hom}}

\newcommand{\Ker}{\operatorname{Ker}}

\renewcommand{\Im}{\operatorname{Im}}
\newcommand{\Coker}{\operatorname{Coker}}

\newcommand{\Spec}{\operatorname{Spec}}

\newcommand{\ab}{{\rm ab}}

\newcommand{\dlog}{\operatorname{dlog}}

\newcommand{\Proj}{\operatorname{Proj}}
\newcommand{\Sm}{\operatorname{\mathbf{Sm}}}

\newcommand{\tr}{{\operatorname{tr}}}

\newcommand{\red}{{\operatorname{red}}}

\newcommand{\Zar}{{\operatorname{Zar}}}
\newcommand{\Nis}{{\operatorname{Nis}}}

\newcommand{\inj}{\hookrightarrow}

\newcommand{\surj}{\rightarrow\!\!\!\!\!\rightarrow}

\newcommand{\id}{{\operatorname{id}}}

\newcommand{\ch}{{\operatorname{ch}}}

\newcommand{\CH}{{\operatorname{CH}}}

\newcommand{\gr}{{\operatorname{gr}}}

\newcommand{\na}{{\rm naive}}

\renewcommand{\epsilon}{\varepsilon}
\renewcommand{\div}{\operatorname{div}}

\def\rmapo#1{\overset{#1}{\longrightarrow}}

\def\piab#1{\pi^{\ab}_1(#1)}
\title[Cycle maps with modulus]{Cycle class maps for Chow groups of zero-cycles with modulus}
\author{Kay R\"ulling \and Shuji Saito}
 
\address{Bergische Universit\"at Wuppertal\\ Gau\ss str. 20, D-42119 Wuppertal, Germany}
\email{ruelling@uni-wuppertal.de}
 
\address{Graduate School of Mathematical Sciences, University of Tokyo, 3-8-1 Komaba, Tokyo 153-8941, Japan}
\email{sshuji@msb.biglobe.ne.jp}

\thanks{K.R.\ was supported by the DFG Heisenberg Grant RU 1412/2-2. 
S.S.\ is supported by the JSPS KAKENHI Grant (20H01791). }
\begin{document}

\begin{abstract}
For a smooth scheme $X$ of pure dimension $d$ over a field $k$ and an effective Cartier divisor $D\subset X$ whose support is a simple normal crossing divisor, we construct a cycle class map
\[ cyc_{X|D}: \CH_0(X|D) \to H^d_\Nis(X,K^M_d(\sO_X,I_D))\]
from the Chow group of zero-cycles with modulus to the cohomology of the relative Milnor $K$-sheaf.
\end{abstract}

\maketitle
\tableofcontents

\section{Introduction}

Let $U$ be a smooth variety of dimension $d$ over a finite field $k$ and $U\hookrightarrow X$ be an open immersion with $X$ normal and proper over $k$ such that $X\backslash U$ is the support of an effective Cartier divisor on $X$.
Let $\piab U$ be the abelianized fundamental group  of $U$.
Kato and Saito \cite{KS-GCFT} established class field theory for $U$ by using the reciprocity map
\[ \rho_U : \varprojlim_D H^d(X_{\Nis},K^M_d(\sO_X,I_D)) \to \piab U,\]
where the limits are over all effective Cartier divisors $D$ supported on 
$X\backslash U$, and $K^M_d(\sO_X,I_D)$ with $I_D\subset \sO_X$ the ideal sheaf of $D$ is the relative Milnor $K$-sheaf on the Nisnevich site over $X$ (see \eqref{para:various2}).
A refinement of this statement was given by Kerz-Saito \cite{Kerz-Saito}\footnote{They needed to assume
$\ch(k)\not=2$. This assumption was removed by Binda-Krishna-Saito \cite{Binda-Krishna-Saito} giving a simpler proof of the main result in \cite{Kerz-Saito} .}
using another reciprocity map 
\[ \phi_U : \varprojlim_D  \CH_0(X|D)  \to \piab U,\]
where $\CH_0(X|D) $ is the Chow group of zero-cycles with modulus.
By definition $\CH_0(X|D) $ is the quotient of the group $Z_0(U)$ of zero-cycles on $U$ by the subgroup generated by divisors of rational functions on curves in $U$ which satisfy a certain modulus conditions with respect to $D$ (see \ref{para:CHM}). 
On the other hand, the Gersten resolution of the Milnor $K$-sheaf yields a natural map 
\eq{intro:CHM3}{\theta :Z_0(U)\longrightarrow H^d(X_{\Nis},K^M_d(\sO_X,I_D)),}
which is surjective by \cite{KS-GCFT}, see \eqref{para:CHM3}.

With the help of ramification theory, the main result of \cite{Kerz-Saito} implies that $\phi_U$ induces a map for a fixed divisor $D$:
\[ \phi_{X|D} :  \CH_0(X|D)  \to \piab {X|D}\]
whose restriction to the degree-zero part is an isomorphism of finite abelian groups.
Here, $\piab {X|D}$ is the quotient of $\piab {U}$ which classifies abelian \'etale coverings of $U$ with ramification over $X\backslash U$ bounded by $D$,
see \cite{Kerz-Saito14}. 
A natural question is whether the latter fact holds with $\CH_0(X|D)$ replaced by
$H^d(X_{\Nis},K^M_d(\sO_X,I_D))$. A positive answer was given by Gupta-Krishna \cite{Gupta-Krishna} assuming $X$ is smooth and projective 
over the finite field $k$.
It implies that there is a natural isomorphism
\begin{equation*}\label{eq1}
 cyc_{X|D}: \CH_0(X|D) \cong H^d_\Nis(X,K^M_d(\sO_X,I_D)),\end{equation*}
which gives a factorization of $\theta$ from \eqref{intro:CHM3}, namely fits into the following commutative diagram
\eq{intro:CD}{
\xymatrix{
& \ar[ld]_{\pi} Z_0(U) \ar[rd]^{\theta} \\
\CH_0(X|D) \ar[rr]^-{cyc_{X|D}} && H^d_\Nis(X,K^M_d(\sO_X,I_D)) \\}}
where $\pi$ is the quotient map.
The question whether such a factorization  of $\theta$ via $\CH_0(X|D)$ exists makes sense in case 
$X$ is  smooth (not necessarily projective) over an arbitrary field $k$ and $D$ is an effective Cartier divisor on $X$. 
It was positively answered in the following cases in which $X$ is always smooth and $D$ is an effective Cartier divisor:
\begin{enumerate}[label=(\arabic*)]
\item\label{intro1} $X$ is a quasi-projective surface  over a field, see \cite[Theorem 1.2]{Krishna15};
\item\label{intro2} $X$ is affine  over an algebraically closed field, see \cite[Theorem 1.3(1)]{Gupta-Krishna20};
\item\label{intro3} $X$ is projective over an algebraically closed field and $D$ is integral, see \cite[Theorem 1.3(2)]{Gupta-Krishna20};
\end{enumerate}
In all  the above cases the cycle map $cyc_{X|D}$ is actually an isomorphism, by \cite{Binda-Krishna-Saito} in the case \ref{intro1} and by \cite{Gupta-Krishna20}
in the cases \ref{intro1}, \ref{intro2}. 
The authors were furthermore informed by Krishna that the existence of the cycle map also
follows from \cite{Krishna18} for  smooth projective varieties over algebraically closed fields if $D$ is reduced.

In this note we prove:

\begin{thm}[Theorem \ref{thm:cycle-map}]\label{intro:cycle-map} 
Let $X$ be a smooth scheme of pure dimension $d$ over a field $k$ and $D\subset X$ be an effective Cartier divisor
 whose support is a simple normal crossing divisor. There exists a map $cyc_{X|D}$ which makes \eqref{intro:CD} commutative.
\end{thm}

We will use the cycle map from the above theorem in \cite{RS-ZNP}, where it is important to allow an arbitrary base field. This is our main motivation.

The proof of Theorem \ref{intro:cycle-map} is different in spirit from the proofs  in \ref{intro1}-\ref{intro3} above and follows instead 
the same strategy as in \cite{RS18}, in which a factorization
\eq{intro:cycleU}{ \CH_0(X|D) \to H^d_\Nis(X,U_{d, X|D})}
is constructed, where $U_{d, X|D}$ is a variant of $K^M_d(\sO_X,I_D)$ (see \ref{para:KM}).
There is another variant $V_{d, X|D}$ and it is easy to prove that $K^M_d(\sO_X,I_D)$
and $V_{d, X|D}$ have the same cohomology group in the top degree $d$ (see \eqref{lem:variousK2}). Thus we are reduced to constructing a factorization
\eq{intro:cycleV}{\CH_0(X|D) \to H^d_\Nis(X,V_{d, X|D}).}
A key input in the construction of \eqref{intro:cycleU} is a projective bundle formula for the cohomology of $U_{d, X|D}$. 
In \S\ref{pbf}, we prove its variant for $V_{d, X|D}$. This is deduced from the original statement by comparing the
 cohomologies of $U_{d, X|D}$ and $V_{d, X|D}$.
To this end we define a decreasing filtration on $V_{d,X|D}$, which ends in $U_{d, X|D+D_{\red}}$ 
(see \ref{para:KM-SNCD}) and compute the  associated graded algebra.
Having this refined projective bundle formula at hand, we construct \eqref{intro:cycleV} in \S\ref{cycle} 
using the same argument as for the construction of \eqref{intro:cycleU} in \cite{RS18}.

\begin{ack}
The authors thank Amalendu Krishna for his comments, in particular for pointing out the references  \cite{Krishna15}, \cite{Krishna18}, and \cite{Gupta-Krishna20}. We thank the referee for his/her many comments 
which helped to improve the exposition.
\end{ack}

\section{The projective bundle formula for relative Milnor {$K$}-theory}\label{pbf}

In this section $k$ is a field of characteristic $p\ge 0$ and $\Sm$ denotes the category of smooth $k$-schemes.

\begin{para}\label{para:KM}
Let $X$ be a $k$-scheme. We denote by $K^M_{r, X}$ ($r \ge 0$) the Nisnevich sheafification 
of the improved Milnor K-theory from \cite{Kerz}. Let $D\subset X$ be a closed subscheme  and denote by 
$j: U:=X\setminus D\inj  X$  (resp. $i:D\inj  X$) the corresponding open (resp. closed) immersion.
Set 
\[\sO_{X|D}^\times:= \Ker(\sO_X^\times\to i_*\sO_D^\times).\]
We consider the following Nisnevich sheaves for $r\ge 1$
\[U_{r, X|D}:=\Im (\sO_{X|D}^\times\otimes_{\Z} j_*K^M_{r-1, U}\to j_*K^M_{r,U}),\quad
V_{r,X|D}:=\Im(\sO_{X|D}^\times\otimes_{\Z}K^M_{r-1, X}\to K^M_{r,X}).\]
Note that 
\begin{enumerate}[label=(\arabic*)]
\item $U_{1, X|D}=V_{1, X|D}=\sO_{X|D}^\times$ and 
the restrictions $K^M_{s,X}\to j_*K^M_{s,U}$, $s=r-1, r$, induce
 a natural map $V_{r, X|D} \to U_{r, X|D}$, which
is injective if $X$ is regular (by the Gersten resolution, see \cite[Proposition 10(8)]{Kerz});
\item\label{para:KM2} if $k$ is an infinite field, then $V_{r,X|D}$ is equal to the Nisnevich sheaf
         $\Ker(K^M_{r, X}\to K^M_{r, D})$ considered in \cite[(1.3)]{KS-GCFT}
         (this follows from \cite[Lemma 1.3.1]{KS-GCFT} and \cite[Proposition 10(5)]{Kerz});
\item the sheaf $U_{r, X|D}$ (for $D$ an effective Cartier divisor) 
    is equal to the Nisnevich sheaf $\sK^M_{r, X|D, \Nis}$ considered in \cite[2.3]{RS18}.
\end{enumerate}
\end{para}



\begin{para}\label{para:Cartier}
Let $X\in \Sm$. 
Denote by $\Omega^i_X$, $i\ge 0$, the Nisnevich sheaf  of absolute differential $i$-forms on $X$ and
set 
\[B_0\Omega^i_X=0, \quad B_1\Omega^i_X=d\Omega^{i-1}_X, \quad 
Z_1\Omega^i_X=\Ker(d:\Omega^i_X\to \Omega^{i+1}_X), \quad Z_0\Omega^i_X=\Omega^i_X.\]
Recall that for $p>0$ we have the  inverse Cartier operator $C^{-1}: \Omega^i_X\to Z_1\Omega^i_X/B_1\Omega^i_X$ at disposal.
It is an isomorphism of sheaves of abelian groups given by
\[C^{-1}(a\dlog b_1\cdots\dlog b_i)=a^p\dlog b_1\cdots\dlog b_i \quad \text{mod }B_1\Omega^i_X,\]
for local sections $a\in \sO_X$, $b_1,\ldots, b_i\in \sO_{X}^\times$. 
For $s\ge 1$, the subsheaves $B_s\Omega^i_X$ and  $Z_s\Omega^i_X$ of $\Omega^i_X$ are recursively defined by 
\[C^{-1}: B_{s-1}\Omega^i_X\xr{\simeq}\frac{B_s\Omega^i_X}{B_1\Omega^i_X},\quad 
C^{-1}: Z_{s-1}\Omega^i_X\xr{\simeq}\frac{Z_s\Omega^i_X}{B_1\Omega^i_X}.\]
In case no confusion can arise  we simply write $B_s$ and $Z_s$ instead of $B_s\Omega^i_X$ and  $Z_s\Omega^i_X$, respectively.
\end{para}

\begin{para}\label{para:local-gr-comp}
Let $X\in \Sm$ and let $D=\div(t)$ be a smooth integral and  principal divisor on $X$.
For $m\ge 1$ and $r\ge 1$ we set 
\[U_{r,m}= U_{r, X|m D},\quad V_{r,m}=V_{r, X|m D}.\]
If $p>0$, write $m=p^s m'$ with $(p,m')=1$ and $s\ge 0$; if $p=0$ we set $s=0$ and $m':=m$.
Set 
\[^mG^r:= 
{\rm Coker}\left(\Omega^{r-2}_D\to \frac{\Omega^{r-1}_D}{B_s}\oplus\frac{\Omega^{r-2}_D}{B_s}, \quad
\alpha\mapsto \big(C^{-s}(d\alpha), (-1)^r m' C^{-s}\alpha\big) \right),\]
where $C^{-s}$ is the $s$-fold iterated inverse Cartier operator, for $s\ge 1$,  and $C^{-0}=\id$.
By \cite[Remark (4.8)]{BK} 
we have an isomorphism
\eq{para:local-gr-comp1}{
^mG^r\xr{\simeq} \frac{U_{r,m}}{U_{r, m+1}}}
given by 
\[\left(a\cdot \dlog\,\ul{b},\, e\cdot \dlog\, \ul{f} \right)\mapsto 
\{1+t^m\tilde{a}, \ul{\tilde{b}}\} + \{1+t^m \tilde{e}, \ul{\tilde{f}}, t\}, \]
where $a, e\in \sO_D$,  $\ul{b}=(b_1,\ldots, b_{r-1})$, and $\ul{f}=(f_1,\ldots, f_{r-2})$, 
$b_i, f_j\in \sO_{D}^\times$ are local sections and  $\tilde{a}$, $\tilde{e}$, $\tilde{b}_i$, $\tilde{f}_j$ 
are  lifts to $\sO_X$. See also \cite[Proposition 2.15, Theorem 2.19]{RS18}, in particular for the sheaf theoretic statement considered here and the case $s=0$ .
\end{para}

\begin{lem}\label{lem-sesUV}
Let the assumption and notation be as in \ref{para:local-gr-comp}.
There is an isomorphism of exact sequences of sheaves on $D$
\eq{lem-sesUV1}{\xymatrix{
0\ar[r] & \frac{V_{r,m}}{U_{r,m+1}}\ar[r] & \frac{U_{r,m}}{U_{r, m+1}}\ar[r] & \frac{U_{r,m}}{V_{r,m}}\ar[r] & 0\\
0\ar[r]&\frac{\Omega^{r-1}_D}{B_s\Omega^{r-1}_D}\ar[u]^{\simeq}\ar[r] & 
           ^mG^r\ar[u]^{\simeq}_{\eqref{para:local-gr-comp1} }\ar[r] &
           \frac{\Omega^{r-2}_D}{Z_s\Omega^{r-2}_D}\ar[u]^{\simeq}\ar[r] & 0.
}}
In particular in characteristic $0$, we have $U_{r,m}=V_{r,m}$.
\end{lem}
\begin{proof}
Note that by the definition of \eqref{para:local-gr-comp1}
the composition 
\[\Omega^{r-1}_D/B_s\to {^mG^r}\to U_{r,m}/U_{r, m+1} \]
factors via a surjection
\[\Omega^{r-1}_D/B_s\surj V_{r,m}/U_{r,m+1}, \quad a\cdot \dlog\ul{b}\mapsto \{1+t^m\tilde{a},\ul{\tilde{b}}\}.\]
Therefore it suffices to show that the lower sequence is exact.
This is immediate for $s=0$. We assume $s\ge 1$ (hence $p>0$).
By \cite[I, Proposition 3.11]{Il} we have  
\[B_{s}\Omega^{r-1}_D = 
F^{s-1}dW_s\Omega^{r-1}_D,\quad  Z_s\Omega^{r-1}_{D}= F^s W_{s+1}\Omega^{r-1}_D,\]
where $W_n\Omega^\bullet_D$ denotes the de Rham-Witt complex from \cite{Il} and $F$ is the Frobenius on it. 
Using the relation $FdV=d$ and $FV=VF=p$, where $V$ is the Verschiebung, we find in $\Omega^{r-1}_D\oplus \Omega^{r-2}_D$
\mlnl{\Big(F^sd\alpha_0 +F^{s-1}d\gamma, (-1)^r m'F^s(\alpha_0)+F^{s-1}d\beta\Big)\\
= \Big(F^s d\alpha, (-1)^r m' F^s(\alpha)+ F^{s-1}d\beta\Big), \quad \text{with }\alpha:=\alpha_0+V(\gamma),}
where $\alpha_0\in W_{s+1}\Omega^{r-2}_D$, $\gamma\in W_s\Omega^{r-2}_D$, and $\beta\in W_s\Omega^{r-3}_D$.
Using that $F^s: W_{s+1}\Omega^{i}_D\to \Omega^{i}_D$ lifts the iterated inverse Cartier operator 
$C^{-s}: \Omega^{i}_D\to \Omega^i_D/B_s$ we see that
the lower row in \eqref{lem-sesUV1} becomes
\[\frac{\Omega^{r-1}_D}{\{F^{s-1}d\gamma\}}\xr{\varphi} 
\frac{\Omega^{r-1}_D\oplus \Omega^{r-2}_D}{\{(F^s d\alpha, 
                                                         (-1)^r m' F^s(\alpha) +F^{s-1}d\beta)\}}\xr{\psi}
\frac{\Omega^{r-2}_D}{\{F^s(\delta)\}},
\]
where,
\[\varphi(x)= (x, 0), \quad \psi(x,y)=y.\]
Note that $\varphi$  clearly surjects onto the kernel of $\psi$.
We show that $\varphi$ is injective.  
Indeed, $\varphi(x)=0$ implies that there exist $\alpha$, $\beta$ with 
$x=F^sd\alpha$ and 
$F^s\alpha=F^{s-1}d\beta$.
We obtain $\alpha-dV\beta\in \Ker(F^s)$. By \cite[I, (3.21.1.2)]{Il},  we find a $\gamma$ such that 
$\alpha- dV\beta = V\gamma$.
Thus
\[x= F^sd\alpha= F^sdV(\gamma)+ F^sddV(\beta)= F^{s-1}d\gamma.\]
Therefore  the lower sequence is exact.
\end{proof}

The following  lemma is direct to check (see \cite[Lemma 2.7]{RS18}).
\begin{lem}\label{lem:lcI}
Let $R$ be a local integral domain with maximal ideal $\fm$ and fraction field $K$.
Let $a,b,c\in R$ and $s,t\in \fm$. The following equalities hold in $K^M_2(K)$
\begin{enumerate}[label=(\arabic*)]
\item\label{lem:lcI1} $\{1+as, 1+bt\}=-\{1+\frac{ab}{1+as}st, -as(1+bt)\}$;
\item\label{lem:lcI2} $\{1+\frac{s-1}{1+ct}ct, 1-\frac{1+ct}{1+cst}s\}=\{1+cst,s\}$
\end{enumerate}
\end{lem}

\begin{para}\label{para:KM-SNCD}
Let $X\in \Sm$  and let $D=\sum_{j=1}^n m_j D_j$ be an effective divisor, such that 
$|D|=\sum_{j=1}^n D_j$ is a simple normal crossing divisor with smooth divisors $D_j$ 
\footnote{We do not assume the  $D_j$ to be connected to include the case where $D=u^*D_0$ for some \'etale map $u:X\to X_0$ and effective divisor $D_0$  on $X_0$ with SNCD support.}
  and $m_j\ge 1$.  For $i\ge 0$  we set 
\[E_i:= D_{i+1}+\ldots+D_n,\]
in particular $E_n=0$,
and denote by $j_i: U_i:=X\setminus|E_i|\inj X$ the open immersion
(with the notation from \ref{para:KM} we have $j=j_0: U=U_0\inj X$).
For $r\ge 1$  we define recursively
\[V_r^{-1}:=0 \quad \text{and}\quad 
V_r^i:=V^{i-1}_r+ 
\Im\left(\sO_{X|D+E_i}^\times\otimes_\Z j_{i*}K^M_{r-1, U_i}\to j_{*}K^M_{r,U}\right), \quad i\ge 0.\]
Since the restriction map $K^M_{r,X}\to j_{*}K^M_{r,U}$ is injective, by \cite[Proposition 10(8)]{Kerz},
 we have $V^i_r\subset V_{r, X|D}$, for all $i$, by Lemma \ref{lem:lcI}\ref{lem:lcI2}.
 We  obtain a filtration
\[U_{r, X|D+E_0}= V_r^0\subset V_r^1\subset\ldots \subset V_r^n= V_{r, X|D}.\]

\end{para}

\begin{para}\label{para:KM-SNCDp}
Let  $r\ge 1$. Let $X\in \Sm$  and let $D=\sum_{j=1}^n m_j D_j$ be as in \ref{para:KM-SNCD}.
If $p>0$, write $m_j=p^{s_j} m_j'$ with $(p,m_j')=1$ and $s_j\ge 0$; if $p=0$ we set
$m_j':=m_j$ and $s_j=0$. (We will use the convention $0^0=1$.)
After renumbering we can assume that 
\eq{para:KM-SNCDp1}{s_1\ge s_{2}\ge \ldots s_n\ge 0.}
Fix $i\in \{1,\ldots,n\}$.
If $p>0$, then there is an inverse Cartier operator on the absolute logarithmic differential forms 
(defined by the same formula as in \ref{para:Cartier})
\[C^{-1}:\Omega^{r-1}_{D_i}(\log E_i)\to \frac{F_*\Omega^{r-1}_{D_i}(\log E_i)}{d\Omega^{r-2}_{D_i}(\log E_i)},\]
where $F$ is the absolute Frobenius and we write 
$\Omega^{r-1}_{D_i}(\log E_i)$ as a  shorthand for  $\Omega^{r-1}_{D_i}(\log E_{i|D_i})$ etc. The $\sO_{D_i}$-submodule 
$B_s\Omega^{r-1}_{D_i}(\log E_i)\subset F^s_*\Omega^{r-1}_{D_i}$ is recursively defined by
\[B_0\Omega^{r-1}_{D_i}(\log E_i):=0, \quad 
C^{-1}: B_{s-1}\Omega^{r-1}_{D_i}(\log E_i)\xr{\simeq}
\frac{B_s\Omega^{r-1}_{D_i}(\log E_i)}{d\Omega^{r-2}_{D_i}(\log E_i)}.\]
Denote by $D(i)$ the minimal divisor satisfying
$p^{s_i} D(i)\ge D+E_i$, i.e., by \eqref{para:KM-SNCDp1}
\[D(i)=\sum_{j=1}^{i} p^{s_j-s_i}m_j'  D_j 
 +\sum_{j=i+1}^n \left\lceil\frac{m_j+1}{p^{s_i}} \right\rceil D_j.\]
We write $\sO_{D_i}(-D):=\sO_X(-D)_{|D_i}$ etc. and have natural inclusions of $\sO_{D_i}^{p^{s_i}}$-modules
\mlnl{\sO_{D_i}(-D(i))^{p^{s_i}}\otimes_{\sO_{D_i}^{p^{s_i}}} B_{s_i}\Omega^{r-1}_{D_i}(\log E_i)\subset
\sO_{D_i}(-D(i))^{p^{s_i}} \otimes_{\sO_{D_i}^{p^{s_i}}} \Omega^{r-1}_{D_i}(\log E_i))\\
\subset \sO_{D_i}(-D-E_i)\otimes_{\sO_{D_i}} \Omega^{r-1}_{D_i}(\log E_i).}
We obtain a sheaf of $\sO_{D_i}^{p^{s_i}}$-modules
\[Q^{i, r-1}_D:=\frac{\sO_{D_i}(-D-E_i)\otimes_{\sO_{D_i}} \Omega^{r-1}_{D_i}(\log E_i)}
{\sO_{D_i}(-D(i))^{p^{s_i}}\otimes_{\sO_{D_i}^{p^{s_i}}} B_{s_i}\Omega^{r-1}_{D_i}(\log E_i)}.\]
If $p=0$, then $s_i=0$ and we have $Q^{i,r-1}_D=\sO_{D_i}(-D-E_i)\otimes_{\sO_{D_i}} \Omega^{r-1}_{D_i}(\log E_i)$.
\end{para}

\begin{lem}\label{lem:Qlc}
Let the notations and assumptions be as in \ref{para:KM-SNCDp}. Assume $k$ is perfect.
\begin{enumerate}[label=(\arabic*)]
\item\label{lem:Qlc01} Then $Q^{i, r-1}_D$ is a locally free $\sO_{D_i}$-module.
\item\label{lem:Qlc02} Let $j_i:U_i= D_i\setminus|E_i|\inj D_i$ be the open immersion.
Then 
\[j_i^*Q^{i,r-1}_D=
 \sO_{U_i}\left(-\sum_{j=1}^{i} p^{s_j-s_i}m_j'  D_j\right)^{p^{s_i}}\otimes_{\sO_{U_i}^{p^{s_i}}} 
 \frac{\Omega^{r-1}_{U_i}}{B_{s_i}\Omega^{r-1}_{U_i}}\]
and the restriction map $Q^{i, r-1}_D\to j_{i*}j_i^*Q^{i,r-1}_D$ is injective.
\end{enumerate}
\end{lem}
\begin{proof}
The first statement in \ref{lem:Qlc02} is  clear, the second statement follows from 
\ref{lem:Qlc01}. Furthermore, \ref{lem:Qlc01} is immediate if $s_i=0$. Thus we assume $s_i\ge 1$ and hence also $p>0$.
 The statement is local. Let $x\in X$ be a point, up to replacing $\{D_j\}_j$ by some ordered subset we can assume that 
 all $D_j$ contain $x$. Hence we  may assume $X=\Spec R$, with $R$ a local ring, essentially smooth over $k$ and with
a regular sequence of parameters $t_1,\ldots, t_N$, where $t_j$ is a local equation for $D_j$, $j=1,\ldots, n$. 
Then $D_i=\Spec R_i$, with $R_i=R/(t_i)$.
Set 
\eq{lem:Qlc0}{\tau:= \prod_{j=1}^i t_j^{m_j}\prod_{j=i+1}^n t_j^{m_j+1},}
which is an  equation for $D+E_i$.
Since $k$ is perfect of positive characteristic, an $R_i^{p^{s_i}}$-basis of 
$\tau\cdot\Omega^{r-1}_{R_i}(\log (t_{i+1}\cdots t_{n}))$
is given by
\[\tau\cdot  t_1^{j_1}\cdots t_N^{j_N}\cdot dt_{k_1}\cdots dt_{k_{r_1}} \cdot
\dlog t_{l_1}\cdots \dlog t_{l_{r_2}},\] 
where   $j_\nu \in[0,p^{s_i}-1]$,
$1\le k_1<\cdots < k_{r_1}\le i$,
$i+1\le l_1<\cdots < l_{r_2}\le n$,
and $r_1$, $r_2\ge 0$ with $r_1+r_2=r-1$.
By \eqref{para:KM-SNCDp1} we have a local equation for $p^{s_i}D(i)$ of the form
\eq{lem:Qlc1}{\sigma^{p^{s_i}}= \tau\cdot \prod_{j=i+1}^n t_j^{h_j}, \quad \text{with } h_j\in [0, p^{s_i}-1].}
It follows from the definition of the inverse Cartier operator, that an $R_i^{p^{s_i}}$-basis
of $\sigma^{p^{s_i}} \cdot B_{s_i}\Omega^{r-1}_{R_i}(\log (t_{i+1}\cdots t_{n}))$
as submodule of $\tau \cdot\Omega^{r-1}_{R_i}(\log (t_{i+1}\cdots t_n))$ is given by
\eq{lem:Qlc2}{\sigma^{p^{s_i}}\cdot 
t_{k_1}^{p^{\mu}-1}\cdots t_{k_{\rho_1}}^{p^{\mu}-1}\cdot dt_{k_1}\cdots dt_{k_{\rho_1}} \cdot
\dlog t_{l_1}\cdots \dlog t_{l_{\rho_2}}, }
where $k_\nu$, $l_\nu$ are as above, $\mu\in[0, s_i-1]$, 
and  $\rho_1\ge 1$, $\rho_2 \ge 0$ with $\rho_1+\rho_2=r-1$.
By \eqref{lem:Qlc1} and  direct inspection of the bases  above, we find that the inclusion 
\[\sigma^{p^{s_i}}\cdot B_{s_i}\Omega^{r-1}_{R_i}(\log (t_{i+1}\cdots t_n))
\inj \tau\cdot\Omega^{r-1}_{R_i}(\log (t_{i+1}\cdots t_n))\]
is a split-injection of $R_i^{p^{s_i}}$-modules.
Hence $Q_D^{i, r-1}$ is a locally free $\sO_{D_i}$-module.
\end{proof}

\begin{lem}\label{lem:Ker-reg}
Let $(R,\fm)$ be a regular local ring and $t_1,\ldots, t_n\in \fm$
 be part of a regular sequence of local parameters ($n\ge 1$). Set $R_n:=R/(t_n)$
 and denote by $\bar{t}_i$ the image of $t_i$ in $R_n$.
Then the following sequence is exact
\[0\to 1+t_n R\to  (R[\tfrac{1}{t_1\cdots t_{n-1}}])^\times
\to (R_n[\tfrac{1}{\bar{t}_1\cdots \bar{t}_{n-1}}])^\times\to 0.\]
Furthermore, $1+t_nR$ is multiplicatively generated by elements $1+t_nu$, with $u\in R^\times$.
\end{lem}
\begin{proof}
By the Auslander-Buchsbaum theorem, $R$ is factorial. 
Since $R/(t_i)$ and $R/(t_n, t_i)$ are again local regular rings (in particular they are integral),
we see that $t_1, \ldots, t_n$ are prime elements in $R$ and 
$\bar{t}_1, \ldots, \bar{t}_{n-1}$ are prime elements in $R_n$.
It follows that every unit in $v\in R[\frac{1}{t_1\cdots t_{n-1}}]$ can be uniquely written in the form
\[v= ut_1^{r_1}\cdots t_{n-1}^{r_{n-1}},\quad \text{with } u\in R^\times, r_1,\ldots, r_{n-1}\in\Z,\]
and similar with $R_n[\frac{1}{\bar{t}_1\cdots \bar{t}_{n-1}}]$.
Hence $v\equiv 1$ mod $t_n$ implies that $r_1=\ldots= r_{n-1}=0$ and $u\in 1+t_n R$.
This implies the exactness of the sequence in the statement.
For the last part observe that if $a\in \fm$, then $1+t_n a=(1+t_n u)\cdot (1+t_nv)$
with $u=a-1$,  $v=1/(1+t_nu)\in R^\times$.
\end{proof}

\begin{prop}\label{prop:grV}
Let the notations and assumptions be as in \ref{para:KM-SNCDp}.
Then there is an isomorphism of  sheaves of abelian groups
\[\theta: Q^{i, r-1}_D \xr{\simeq} \frac{V_r^i}{V_r^{i-1}}\]
given on local sections  by 
\[\theta(a\dlog b_1\cdots \dlog b_{r-1})=\{1+\tilde{a}, \tilde{b}_1,\ldots, \tilde{b}_{r-1}\},\]
where $\tilde{a}\in \sO_X(-D-E_i)$, $\tilde{b}_j \in \sO_{X\setminus E_i}^\times$ are local lifts of
$a$ and $b_j$, respectively.
\end{prop}
\begin{proof}
For an \'etale $X$-scheme $W$ let $T(W)$ denote the free abelian group on the generators
$[e,f_1,\ldots, f_{r-1}]$, with $e\in \Gamma(W, \sO_X(-D-E_i)_{|W})$ and 
$f_j\in \sO^\times(W\setminus E_{i|W})$, $j=1,\ldots, r-1$.
Then $W\mapsto T(W)$ defines a Nisnevich sheaf on $X$ 
and we have  surjective morphisms of sheaves
\eq{prop:grV01}{T \to Q^{i, r-1}_D, \quad
[e, f_1,\ldots, f_{r-1}]\mapsto \bar{e}\dlog \bar{f}_1\cdots \dlog\bar{f}_{r-1},}
where $\bar{e}$ and $\bar{f}_j$ denote the image of $e$ and $f_j$ in $\sO_{D_i}(-D-E_i)$ and 
$\sO_{D_i\setminus E_i}^\times$, respectively,
and 
\eq{prop:grV02}{ T \to \frac{V_r^i}{V_r^{i-1}}, \quad
[e, f_1,\ldots, f_{r-1}]\mapsto \{1+e, f_1,\ldots, f_{r-1}\}.}
The map $\theta$ from the statement is well-defined and surjective when we show that the kernel of the surjection
$\eqref{prop:grV01}$ is mapped to zero under \eqref{prop:grV02}.
This is a local question and it suffices to consider stalks.
Note that both sheaves in the statement have support in $D_i$. 
Therefore it suffices to show  the statement in every point $x\in D_i$.
Up to  replacing $\{D_j\}_j$ by an ordered subset we may assume that all components of $D$ meet in $x$.
Thus we can assume $X=\Spec R$, where $R$ is a local ring with maximal ideal $\fm$, which is
essentially smooth over $k$. Let $t_j\in\fm$ be a local equation 
of $D_j$, $j=1,\ldots n$. Set $R_i:=R/(t_i)$.

We first show that a well-defined map $\theta$ as in the statement exits.
Set $M^{\rm gp}:=R_i[\frac{1}{t_{i+1}\cdots t_n}]^\times$ and 
denote by $\tau$ the  equation \eqref{lem:Qlc0} for $D+E_i$. There is a well-defined map
\eq{prop:grV1}{\tau R_i \oplus \bigoplus_{j=1}^{r-1} M^{\rm gp}\to \frac{V_r^i}{V_r^{i-1}}}
given by 
\[(a, (b_1,\ldots, b_{r-1}))\mapsto \{1+\tilde{a}, \tilde{b}_1,\ldots, \tilde{b}_{r-1}\},\]
where $\tilde{a}\in \tau R$, $\tilde{b}_j\in (R[\frac{1}{t_{i+1}\cdots t_n}])^\times$ are lifts of $a$ and $b_j$,
respectively.
We have to check the independence  of the choices:
Choosing a different lift of $a$  results in multiplying $1+\tilde{a}$ by an element
$(1+t_i \tau_i c)$, for $c\in R$. The independence of the choice of $\tilde{a}$
follows, since $\{1+t_i\tau c, \ul{\tilde{b}}\}\in V^{i-1}_{r}$. 
Two different lifts of $b_1$ differ by Lemma \ref{lem:Ker-reg} by a product of element of the form
$(1+t_i e)$ with $e\in R^\times$. For $c\in R$ and $\tilde{b}_j$ as above, we compute
\begin{align*}
\{1+\tau c, 1+t_i e, \ul{\tilde{b}}\} &= -\{1+e t_i , 1+c \tau , \ul{\tilde{b}}\}\\
                                                        & = -\{1+ \frac{ec}{1+et_i} t_i\tau, -et_i(1+c\tau), \ul{\tilde{b}}\}, & & 
                                                                       \text{by }\ref{lem:lcI}\ref{lem:lcI1},\\
                                                        &\in V^{i-1}_r.
\end{align*}
This shows the independence of the choice of the lift of $\tilde{b}_1$ and similarly for $\tilde{b}_j$.
Clearly \eqref{prop:grV1} is $\Z$-linear in the $b_j$; the linearity in $a$ follows from
\[(1+\tilde{a}_1)(1+\tilde{a}_2)= (1+\tilde{a}_1+\tilde{a}_2) (1+ t_i c),\]
for $\tilde{a}_i$, $c\in\tau R$.
Furthermore, if $p\neq 2$, then
\[\{1+\tau c, \ul{\tilde{b}}\}\equiv  2\cdot \{1+ \tfrac{1}{2}\tau c, \ul{\tilde{b}}\} \quad \text{mod } V^{i-1}_r.\]
Thus the formula $\{b,b\}=\{b,-1\}$ in $K^M_2(k(X))$ implies that \eqref{prop:grV1}
is alternating in the $b$'s (also for $p=2$). Altogether we see that \eqref{prop:grV1}
induces a surjective map
\eq{prop:grV2}{\tau R_i\otimes_\Z \bigwedge_{j=1}^{r-1} M^{\rm gp}\surj \frac{V_r^i}{V_r^{i-1}}.}
We claim that this map factors via $\tau\cdot \Omega^{r-1}_{R_i}(\log( t_{i+1}\cdots t_n))$.
By \cite[(1.7)]{KatoLog} we have to show that \eqref{prop:grV2} maps the following elements to zero
\eq{prop:grV3}{\tau a\otimes a\wedge b_1\wedge\cdots\wedge b_{r-2}- 
\sum_\nu \tau u_\nu\otimes u_\nu\wedge b_1\wedge\cdots\wedge b_{r-2},}
where $a\in M=M^{\rm gp}\cap R_i$ and $u_\nu\in R_i^\times$ such that 
$a=\sum_\nu u_\nu$, and $b_j\in M^{\rm gp}$.
Choose lifts $\tilde{u}_\nu\in R^\times$ of $u_\nu$ and set 
$\tilde{a}=\sum_\nu \tilde{u}_\nu$.
Let $\ul{\tilde{b}}=(\tilde{b}_1,\ldots, \tilde{b}_{r-2})$ with 
$\tilde{b}_j\in R[\frac{1}{t_{i+1}\cdots t_n}]^\times$ lifts of $b_j$.
Write 
\[(1+\tau\sum_\nu \tilde{u}_\nu)= \prod_\nu (1+\tau \tilde{u}_\nu) \cdot (1+\tau t_i e),\]
for some $e\in R$. Assume first $e\in R^\times$ is a  unit. 
Then in $K^M_r(k(X))$
\begin{align*}
\{1+\tau\tilde{a}, \tilde{a},\ul{\tilde{b}}\}& =-\{1+\tau\tilde{a}, -\tau, \ul{\tilde{b}}\}\\
                                                         &=-\sum_\nu \{1+\tau \tilde{u}_\nu, -\tau, \ul{\tilde{b}}\} 
                                                             - \{1+\tau t_i e,-\tau, \ul{\tilde{b}}\}\\
                                                         &= \sum_\nu \{1+\tau \tilde{u}_\nu, \tilde{u}_\nu, \ul{\tilde{b}}\} 
                                                             +\underbrace{\{1+\tau t_i e, t_i e,\ul{\tilde{b}}\}}_{\in V^{i-1}_r}.
\end{align*}
Since the starting term and the final term are in $V^i_r$,
the equality of these two terms holds in $V^i_r$. 
If $e$ is not a unit we can write 
\[(1+\tau t_i e)=(1+\tau t_i (e-1))\cdot (1+\tau t_i \tfrac{1}{1+\tau t_i (e-1)})\]
and argue similarly.
Hence \eqref{prop:grV2} sends the relation \eqref{prop:grV3} to zero and
induces a surjection
\eq{prop:grV4}{\tau \cdot \Omega^{r-1}_{R_i}(\log (t_{i+1}\cdots t_n))\surj \frac{V^i_r}{V^{i-1}_r}.}
Assume $s_i\ge 1$. Let $\sigma^{p^{s_i}}$ be the equation \eqref{lem:Qlc1} for $p^{s_i}D(i)$.
We want to show that \eqref{prop:grV4} maps 
$\sigma^{p^{s_i}}\cdot B_{s_i}\Omega^{r-1}_{R_i}(\log (t_{i+1}\cdots t_n))$ to zero.
The latter module is generated as  abelian group by the elements (cf. \cite[0, Proposition 2.2.8]{Il})
\eq{prop:grV5}{(\sigma^{p^{s_i-\nu}} u)^{p^\nu}\dlog u\dlog v_1\cdots\dlog v_{r-2}, 
\quad u\in R_i^\times, v_j\in M^{\rm gp}, \, 0\le \nu\le s_i-1.}
We compute in $K^M_r(k(X))$ (cf. \cite[Lemma (4.5)]{BK})
\begin{align*}
\{1+(\sigma^{p^{s_i-\nu}}\tilde{u})^{p^\nu}, \tilde{u}, \ul{\tilde{v}}\} &=
p^{\nu}\{1+\sigma^{p^{s_i-\nu}}\tilde{u}, \tilde{u}, \ul{\tilde{v}}\}\\
&= - p^{\nu}\{1+\sigma^{p^{s_i-\nu}}\tilde{u}, -\sigma^{p^{s_i-\nu}}, \ul{\tilde{v}}\}\\
&=-\{1+\sigma^{p^{2 s_i-\nu}}\tilde{u}^{p^{s_i}}, -\sigma, \ul{\tilde{v}}\}\\
& =\{1+\sigma^{p^{2s_i-\nu}-1}w_1, w_2, \ul{\tilde{v}}\}, & & \text{with } w_j\in R^\times,
\end{align*}
where the last equality holds by Lemma \ref{lem:lcI}\ref{lem:lcI2}.
Since  $p^{2s_i-\nu}-1\ge 2$, for all $\nu\le s_i-1$, the  last term lies in $V_{r}^{i-1}$.
Thus \eqref{prop:grV4} maps the element \eqref{prop:grV5} to zero.
We obtain a surjective morphism $\theta$ as in the statement.
Set $S_i:=R_i[\frac{1}{t_1\cdots \widehat{t_i}\cdots t_n}]$ (we omit $t_i$ in the denominator).
Consider the following diagram
\[\xymatrix{
Q^{i,r-1}_D(R_i)=\frac{\tau\cdot \Omega^{r-1}_{R_i}(\log (t_{i+1}\cdots t_n))}
{\sigma^{p^{s_i}}\cdot B_{s_i}\Omega^{r-1}_{R_i}(\log(t_{i+1}\cdots t_n))}
\ar[rr]^-{\frac{1}{t_i^{m_i}}\cdot\text{restr.}}\ar[d]^{\theta} & &
\frac{\Omega^{r-1}_{S_i}}{B_{s_i}\Omega^{r-1}_{S_i}}\ar@{_(->}[d]^{\eqref{lem-sesUV1}}\\
\frac{V^i_r(R_i)}{V^{i-1}_r(R_i)}\ar[rr]^-{\text{restriction}} & &
\frac{V_{r, m_i}(S_i)}{U_{r, m_i+1}(S_i)}.
}\]
The diagram is commutative and  the vertical arrow on the right hand side is injective by Lemma  \ref{lem-sesUV}.
We can write $R$ as a directed limit of rings $A$ which are smooth over $\F_p$ and on which $t_1\cdots t_n$
defines a simple normal crossing divisor. For $R$ replaced by such an $A$ 
the top horizontal map is injective by Lemma \ref{lem:Qlc}; since directed limits are exact this holds for $R$ as well. 
Hence $\theta$ is injective, which yields the statement.
\end{proof}

\begin{rmk}\label{rmk:theta-func}
Let $X$, $D$, $D_j$, $E_i$, and $V^i_r=V^i_{r, X|D}$ be as in \ref{para:KM-SNCD} and $Q_D^{i,r-1}$ be as in \ref{para:KM-SNCDp}. Let 
$f:Y\to X$ be a smooth morphism and define $V^i_{r, Y|f^*D}$  and $Q_{f^*D}^{i,r-1}$ similarly, 
where we replace $D_j$ by $f^*D_j$ (which is smooth but possibly not connected)
and $E_i$ by $f^*E_i$.  We obtain a commutative diagram on $X_\Nis$
\[\xymatrix{
f_*Q^{i, r-1}_{f^*D}\ar[r]^-{\theta} & f_*(V^i_{r-1, Y|f^*D}/V^i_{r, Y|f^*D})\\
Q^{i, r-1}_D\ar[u]\ar[r]^-{\theta} &  V^i_{r-1, X|D}/V^i_{r, X|D},\ar[u]
}\]
where the vertical maps are the natural pullback maps and the horizontal maps are the isomorphisms constructed in Proposition \ref{prop:grV}.
By the explicit formula for $\theta$ the diagram obviously commutes.
\end{rmk}

The following corollary is not used in this paper but we include it here for use in \cite{RS-HLS}. 

\begin{cor}\label{lem:VZN}
Let $X\in \Sm$ with $\dim X=d$ and let $D$ be an effective Cartier divisor with simple normal crossing support.
Set $\sG:=V_{r, X|D}/V_{r, X|D+D_{\red}}$, $r\ge 1$, which we view as a sheaf on $Y=D_\red$.
Then the natural morphism $H^{d-1}(Y_\Zar,\sG)\to H^{d-1}(Y_{\Nis}, \sG)$ is surjective.
\end{cor}
\begin{proof}
With the notation from \ref{para:KM} we set 
\[\sF:=\frac{U_{r, X|D+D_{\red}}}{U_{r, X|D+2D_{\red}}}\quad \text{and}\quad \sH:=\frac{V_{r, X|D}}{U_{r, X|D+D_{\red}}}.\]
We have an exact sequence of Nisnevich sheaves
\eq{lem:VZN1}{\sF\to \sG\to \sH\to 0.}
By Grothendieck-Nisnevich vanishing we obtain exact sequences for $\tau\in\{\Zar,\Nis\}$
\[H^{d-1}(Y_\tau, \sF)\to H^{d-1}(Y_\tau,\sG) \to H^{d-1}(Y_\tau,\sH)\to 0. \]
By \ref{para:KM-SNCD} and Proposition \ref{prop:grV} the quotient $\sH$ admits a finite decreasing filtration 
whose  successive quotients are coherent $\sO_Y$-modules. Hence  the natural morphism
$H^{d-1}(Y_\Zar, \sH)\to H^{d-1}(Y_\Nis,\sH)$ is an isomorphism.
By \cite[Proposition 2.15, Theorem 2.19]{RS18} (cf. also \ref{para:err1} below)  the same holds for $\sF$.
This yields the statement.
\end{proof}

\begin{thm}\label{thm:proj-pf-relK}
Let $X$ be a smooth $k$-scheme  of pure dimension $d$ and $D$ an effective divisor on $X$
 whose support has  simple normal crossings. Let $E$ be a vector bundle of dimension $N+1$ on $X$
and denote by $\pi: \P:=\Proj({\rm Sym}\, E^\vee)\to X$ the projection from the structure map.
Let 
\[\eta^t:=c_1(\sO_{\P}(1))^t\in H^t_{\Nis}(\P, K^M_{t,\P})=\Hom_{D(\P_\Nis)}(\Z, K^M_{t,\P}[t]), \quad
t=0\ldots, N.\]
The map  in $D(X_\Nis)$, the derived category of abelian Nisnevich sheaves on $X$, 
\eq{prop:proj-pf-relK2}{-\cup \eta^t: \bigoplus_{t=0}^N V_{r-t,X|D}[-t]
\xr{\simeq} R\pi_* V_{r, \P|\pi^*D}, \quad r\ge 0,}
is an isomorphism.
If $D_{\rm red}$ is smooth the same is true  for the Zariski site.
\end{thm}
\begin{proof}
Write $D=\sum_{j=1}^n m_j D_j$ as in \ref{para:KM-SNCD}.
It suffices to show that $\cup \eta^t$ induces an isomorphism
$V_{r-t, X|D}\xr{\simeq}R^t\pi_*V_{r, \P|\pi^*D}$, for all $0\le t\le N$,
and $R^t\pi_*V_{r, \P|\pi^*D}=0$, for all $t> N$.
This is a local question and we may therefore assume $\P=\P^N_X$.
The corresponding statement for $U_{r, X|D+D_\red}$ holds by 
\cite[Theorem 2.28]{RS18} (see also \ref{para:err2}) below.
By Proposition \ref{prop:grV} (and Remark \ref{rmk:theta-func}) we are reduced to show:
\begin{enumerate}
\item $\cup\dlog(\eta^t): Q_{D}^{i, r-t-1}\to R^t\pi_* Q_{\pi^*D}^{i, r-1}$ is an isomorphism, for all $0\le t\le N$ 
and $1\le i\le n$;
\item $R^t\pi_* Q_{\pi^*D}^{i, r-1}=0$, for all $t>N$ and all $1\le i\le n$,
\end{enumerate}
where $Q_{D}^{i,r-1}$ is defined in \ref{para:KM-SNCDp}.
By the definition of $Q^{i,r-1}_D$ and the projection formula (e.g. \cite[Proposition (0.12.2.3)]{EGAIII}),
this follows from the corresponding statements
for $\Omega^{r-1}_{D_i}(\log E_i)$ and - if $p>0$ -  for $B_s\Omega^{r-1}_{D_i}(\log E_i)$
(where $E_i=D_{i+1}+\ldots+D_n$). These are well-known, cf. \eqref{no2.25.2} and 
the proof of \eqref{no2.25.4} (in the case char$(k)>0$) below.
\end{proof}

\begin{defn}\label{defn:proj-tr}
Let the notations and assumptions be as in Theorem \ref{thm:proj-pf-relK}.
We define 
\[\tr_\pi: R\pi_* V_{N+r,\P|\pi^*D}[N]\to V_{r, X|D}\]
to be the composition of the natural map $R\pi_* V_{N+r,\P|\pi^*D}[N]\to R^N\pi_* V_{N+r,\P|\pi^*D}$
followed by the inverse of the isomorphism $-\cup \eta^N: V_{r, X|D} \xr{\simeq} R^N\pi_* V_{N+r,\P|\pi^*D}$.
\end{defn}

\subsection{Small erratum to {\cite{RS18}}}
We use the opportunity to slightly correct an assumption and a proof from \cite{RS18},
see \ref{para:err1} and \ref{para:err2} below.
Please note that these corrections do not have any effect on the rest of {\em loc. cit.},
in particular not on any of the main results.

\begin{para}\label{para:err1}
In \cite[2.1.4]{RS18} we consider the following situation:
Let $X$ be a smooth and separated scheme over  $k$.
Let $\sum_{\lambda\in \Lambda} D_\lambda$
be a simple normal crossing divisor (with smooth components $D_\lambda$).
For $\fm=(m_\lambda)_{\lambda\in \Lambda}\in \N^{|\Lambda|}$
we define $D_\fm=\sum_\lambda m_\lambda D_\lambda$ and set for $\nu\in \Lambda$
\[\gr^{\nu}_{D_\fm} K^M_{r, X}:= U_{r, X|D_\fm}/U_{r, X|D_{\fm+\delta_\nu}},\]
where $\delta_\nu=(0, \ldots, 0, 1, 0,\ldots ,0)$ with $1$ in the  $\nu$'s place.
(In {\em loc. cit.} $U_{r, X|D}$ is denoted by $\sK^M_{r, X|D}$.)
In the Propositions 2.14, 2.15, Theorem 2.19, and Corollary 2.20 of \cite{RS18} 
we give formulas for $\gr^{\nu}_{D_\fm}K^M_{r,X}$.
In {\em loc. cit.} we allow also some of the $m_\lambda$ to be zero and in this case the formulas are not correct as written.
However, if  $m_\lambda\ge 1$, for all $\lambda\in \Lambda$, these formulas are correct and this is the only case used
in the rest of \cite{RS18}. (The formulas are used to understand the difference between 
$U_{r, X|D}$ and $U_{r, X|D_\red}$; the latter group can then be studied using 
Corollary 2.10,  which is correct as written, i.e., the $m_\lambda$ may be zero.)
\end{para}

\begin{para}\label{para:err2}
Let the notation and assumptions be as in \cite[Lemma 2.25]{RS18}.
We correct the proof of the isomorphisms 
\eq{no2.25.1}{-\cup \dlog(c_1(\sO(1))^i): \omega^{q-i}_{X|D, \fm,\nu}/B^{q-i}_{X|D, r,\fm,\nu}\xr{\simeq}
R^i\pi_{D_\nu*} (\omega^q_{P_X|P_D,\fm,\nu}/B^q_{P_X|P_D, r,\fm,\nu}),}
for $i\ge 0$. To this end replace on page 1013 everything after  
``Now we prove the statement for $\omega^q_{\fm ,\nu}$'' in 
line 8 until ``\ldots from Lemma 2.26 below.'' in line 25 by the following: 

We have a well-known isomorphism 
\eq{no2.25.2}{-\cup \dlog(c_1(\sO(1))^i): \Omega^{q-i}_X(\log D)\xr{\simeq}R^i\pi_{X*} \Omega^q_{P_X}(\log P_D).}
(Without log-poles this can, e.g., be deduced from \cite[Exp XI, Thm 1.1]{SGA7II}; the case with log-poles then 
follows by considering the degree-wise weight filtration. More precisely, let $w_j=w_j\Omega^q_X(\log D)$ be the $\sO_X$-submodule of $\Omega^q_X(\log D)$
whose sections have poles along at most $j$ different components of $D$. The Poincar\'e residue induces an isomorphism 
$w_j/w_{j-1}\cong \Omega^{q-j}_{D^j}$, where $D^j=\sqcup_{i_1<\ldots<i_j} D_{i_1}\cap\ldots\cap D_{i_j}$, 
cf. \cite[(3.1.5.2)]{HodgeII}. This isomorphism is compatible with the pullback to $P_X$ and we can deduce the  isomorphism \eqref{no2.25.2}.)
Using projection formula and base-change \eqref{no2.25.2} yields an isomorphism
\eq{no2.25.3}{-\cup \dlog(c_1(\sO(1))): \omega^{q-i}_{X|D,\fm,\nu}
\xr{\simeq}R^i\pi_{X*} \omega^q_{P_X|P_D,\fm,\nu},  \quad i\ge 0.}
It remains to show
\eq{no2.25.4}{-\cup \dlog(c_1(\sO(1))): B^{q-i}_{X|D, r, \fm,\nu}
\xr{\simeq}R^i\pi_{X*} B^q_{P_X|P_D,r, \fm,\nu}, \quad i\ge 0.}
By a limit argument we can reduce to the case that $k$ is finitely generated over its prime field,
so that $\Omega^q_X(\log D)=0$, for $q$ large enough.

{\em First case: ${\rm char}(k)=0$.} 
In this case the statement follows by descending induction on $q$ from the exact sequence
\eq{no2.25.5}{0\to Z^{q}_{X|D, \fm,\nu}\to \omega^q_{X|D,\fm ,\nu}\to B^{q+1}_{X|D,\fm,\nu}\to 0,}
the equality $Z^{q}_{X|D, \fm,\nu}=B^{q}_{X|D, \fm,\nu}$ (see \cite[Lemma 6.2]{BS19}),  and  \eqref{no2.25.3}.
(Lemma 2.26 is not  needed.)

{\em Second case: ${\rm char}(k)=p$.} 
For $r=1$ the statement follows by descending induction on $q$ from the 
exact sequence \eqref{no2.25.5} together with the exact sequence 
\[0\to B^{q}_{X|D,\fm,\nu}\to Z^{q}_{X|D,\fm,\nu}\to \omega^q_{\fm',\nu}\to 0,\]
coming from the Cartier isomorphism \cite[Theorem 2.16]{RS18}. 
For higher $r$ it follows from the recursive definition of $B^q_{X|D, r,\fm,\nu}$ in \cite[2.4.4]{RS18}.
\end{para}

\subsection{Comparison of different relative Milnor {$K$}-sheaves}\label{subsec:variousK}
We will have to use results from \cite{KS-GCFT} in which a slightly different version of the relative Milnor 
$K$-theory is used. Here we observe that these different versions give rise to the same top degree Nisnevich cohomology,
so that the difference won't play a role for us.

\begin{para}\label{para:variousK}
Let $X$ be a $k$-scheme and $D\subset X$ a nowhere dense closed subscheme. 
Denote by 
\[K^{M,\na}_{r, X}=(\sO_X^{\times})^{\otimes_{\Z} n}/J,\]
the naive Milnor $K$-sheaf on $X_\Nis$, where $J\subset (\sO_X^{\times})^{\otimes_{\Z} r}$ 
denotes the abelian subsheaf locally generated by elements 
of the form $a_1\otimes\ldots\otimes a_r$,  all $a_i\in \sO_X^\times$, and $a_i+a_j=1$, for some $i\neq j$.
By \cite[Proposition 10 and Theorem 13]{Kerz} there is a surjective map
\eq{para:various1}{K^{M,\na}_{r, X}\surj K^M_{r, X}}
and it is an isomorphism at the stalks $x$ with infinite residue field $k(x)$.
In \cite[(1.3)]{KS-GCFT} the relative Milnor $K$-sheaf of $(X, D)$ is defined by the formula
\eq{para:various2}{K^{M}_r(\sO_X, I):= \Ker(K^{M,\na}_{r, X}\to i_*K^{M,\na}_{r,D}),}
where $I$ is the ideal sheaf of $D$ and $i: D\inj X$ denotes the closed immersion. 
\end{para} 

\begin{lem}\label{lem:variousK}
Let the notations and assumptions be as  in \ref{para:variousK}.
Assume $X$ is noetherian, reduced, and of pure dimension $d<\infty$.
Then \eqref{para:various1} induces a surjection
\eq{lem:variousK1}{K^{M}_r(\sO_X, I)\surj V_{r, X|D},}
and the induced map on top-degree Nisnevich cohomology is an isomorphism
\eq{lem:variousK2}{H^d(X_{\Nis}, K^{M}_r(\sO_X, I))\xr{\simeq} H^d(X_\Nis, V_{r, X|D}).}
\end{lem}
\begin{proof}
The surjection in \eqref{lem:variousK1} holds by \cite[Lemma 1.3.1]{KS-GCFT}.
By \cite[Proposition 10(4)]{Kerz} the map \eqref{para:various1} is an isomorphism on fields.
Hence \eqref{lem:variousK1} has kernel supported in codimension $\ge 1$.
The isomorphism in \eqref{lem:variousK2} follows therefore from Grothendieck-Nisnevich vanishing.
\end{proof}

\section{The cylce class map}\label{cycle}
In this section $k$ denotes a field.  For a scheme $Z$ we denote by $Z_{(i)}$ (resp. $Z^{(i)}$) the set 
of points of $Z$ whose closure have (co)dimension $i$.

\begin{para}\label{para:CHM}
Let $X$ be an equidimensional $k$-scheme of finite type  and $D$ an effective Cartier divisor on $X$ such that 
$U=X\setminus |D|$ is smooth over $k$. Set $d=\dim X$.
For $C\subset X$ an integral curve  not contained in the support of $D$ 
and  with normalization $\nu:\tC\to C$, we set
\begin{equation}\label{para:CHM1}
G(C,D):= \bigcap_{x\in \tC\cap \nu^{-1}(D)} \Ker\big(\sO_{\tC,x}^\times 
\to \sO_{\tC\times_C D,x}^\times\big)\;\subset k(C)^\times.
\end{equation}
We define a map
\[  \partial_C: G(C, D) \rmapo{\div_{\tC}} Z_0(\tC) \rmapo{\nu_*} Z_0(U),\]
and put
\[ \CH_0(X|D)=\Coker\left(\partial=\sum_C\partial_C :
\underset{C}{\bigoplus}\;G(C,D) \to Z_0(U)\right),\]
where the sum is over all $C\subset X$ as above and $Z_0(Y)$ denotes the group of zero-cycles on $Y$.

For a closed point $x\in U$ the Gersten resolution (\cite[Proposition 10(8)]{Kerz})  yields an isomorphism 
\eq{para:CHM2}{\theta_x: \Z\xr{\simeq} H^d_x(U_{\Nis}, K^M_{d, U})
\cong H^d_x(X_{\Nis}, V_{d, X|D}),}
where we use the notation from \ref{para:KM}.
By \cite[Theorem 2.5]{KS-GCFT} and Lemma \ref{lem:variousK} we obtain a surjective map
\eq{para:CHM3}{\theta=\sum_x \theta_x: Z_0(U)=\bigoplus_{x\in U_{(0)}} \Z
\surj H^d(X_{\Nis}, V_{d, X|D}).}
Similarly, we can define a map $Z_0(U)\to H^d(X_{\Nis}, U_{d, X|D})$.
As a special case of \cite[Proposition 3.3]{RS18} this map factors via $\CH_0(X|D)$, if $X$ is smooth and 
$D$ has simple normal crossing support.
The following theorem is a refinement of this statement with $U_{d, X|D}$ replaced by $V_{d,X|D}$.
It implies Theorem \ref{intro:cycle-map}, by  \eqref{lem:variousK2}.
\end{para}

\begin{thm}\label{thm:cycle-map}
Let the notation and assumptions be as in \ref{para:CHM} above.
We assume additionally that  $X$ is smooth and that the support of $D$ is a simple normal crossing divisor.
Then $\eqref{para:CHM3}$ factors to give a surjective  map
\[\CH_0(X|D)\surj H^d(X_{\Nis}, V_{d, X|D}).\]
\end{thm}

See the introduction for a discussion on under which  assumptions on $X, D, k$ the theorem was previously known (in view of \eqref{lem:variousK2}).
Before we can prove the theorem, we need to recall some notations and results from \cite{RS18}.
\begin{para}\label{para:coh-symb}
Let $Y=\Spec A$ be an affine scheme. Let $s_1,\ldots, s_c\in A$  and set  $Z=\Spec A/(s_1,\ldots, s_c)$.
Then $\fV=\{V_1,\ldots, V_c\}$ with $V_i=\Spec A[\frac{1}{s_i}]$ is an open covering of $Y\setminus Z$.
Let $F$ be a sheaf of abelian groups on $Y_\Nis$ and denote by $C^\bullet(\fV, F)$ the \v{C}ech complex of $F$ with respect to
$\sV$. We obtain natural maps
\ml{para:coh-symb1}{F(V_1\cap\ldots\cap V_c)=C^{c-1}(\fV, F) \to H^{c-1}(C^\bullet(\fV, F))\\
\to 
H^{c-1}((Y\setminus Z)_{\Zar}, F)\to H^{c-1}((Y\setminus Z)_{\Nis}, F)\to
H^c_Z(Y_\Nis, F),}
where the last map is the boundary map of the localization sequence.
For $a\in F(V_1\cap\ldots\cap V_c)$ we denote by 
\[\genfrac{[}{]}{0pt}{}{a}{s_1,\ldots, s_c}\in H^c_Z(Y_\Nis, F)\]
the image of $a$ under \eqref{para:coh-symb1}.
We note the following two obvious functoriality properties of this symbol:
\begin{enumerate}[label=(\arabic*)]
\item\label{para:coh-symb2} if $h: F\to G$ is a morphism of sheaves then
\[h\left(\genfrac{[}{]}{0pt}{}{a}{s_1,\ldots, s_c}\right)=\genfrac{[}{]}{0pt}{}{h(a)}{s_1,\ldots, s_c}\quad 
\text{in } H^c_Z(Y_\Nis, G);\]
\item\label{para:coh-symb3} if $j: U\inj Y$ is an affine open immersion, then 
\[j^*\left(\genfrac{[}{]}{0pt}{}{a}{s_1,\ldots, s_c}\right)=\genfrac{[}{]}{0pt}{}{j^*a}{j^*s_1,\ldots, j^*s_c}\quad 
\text{in } H^c_{Z\cap U}(U_\Nis, F_{|U}).\]
\end{enumerate}

Now assume $Y$ is a regular $k$-scheme of pure dimension $d$ and $y\in Y^{(d-1)}$.
In this case the Gersten complex of $K^M_{d,Y}$ is a flasque resolution and we obtain an isomorphism
\eq{para:coh-symb4}{k(y)^\times\xr{\simeq}H^{d-1}_y(Y_\Zar, K^M_{Y,d})=H^{d-1}_y(Y_\Nis, K^M_{Y,d})=
H^{d-1}_y(Y_{(y), \Nis}, K^M_{Y,d}), }
where $Y_{(y)}=\Spec \sO_{Y,y}$. 
Let $s_1,\ldots, s_{d-1}\in \sO_{Y,y}$ be a regular sequence of parameters. 
Under the isomorphism \eqref{para:coh-symb4} a function  $f\in k(y)^\times$ is mapped  to the symbol
\[\pm\genfrac{[}{]}{0pt}{}{\{\tilde{f}, s_1,\ldots, s_{d-1}\}}{s_1,\ldots, s_{d-1}},\]
where $\tilde{f}\in \sO_{Y,y}$ is any lift of $f$, see \cite[Corollary 2.3]{RS18}.
\footnote{
The sign depends on the choice of sign in the definition of the tame symbols which appear in the Gersten resolution.}
(To see this one considers the \v{C}ech complex of the Gersten resolution of $K^M_{Y,d}$. 
This yields a double complex whose associated complex  comes 
with  natural augmentation maps from the Gersten complex and  the \v{C}ech complex of $K^M_{Y,d}$, 
which both are quasi-isomorphisms. If one follows the image of $f$ under the induced isomorphisms on cohomology
one obtains the above description, for details see {\em loc. cit.})
\end{para}


\begin{proof}[Proof of Theorem \ref{thm:cycle-map}.]
The proof is similar to \cite[Proposition 3.3]{RS18}, though we will need the  projective bundle formula
for $V_{r, X|D}$ proved in the previous section and we take a short cut around the Cousin resolution used in 
{\em loc. cit.}

We have the spectral sequence
\[E^{i,j}_1 =\bigoplus_{x\in X^{(i)}} H^{i+j}_x(X_\Nis, V_{d, X|D}) 
\Longrightarrow H^{i+j}(X_{\Nis}, V_{d, X|D}).\]
Grothendieck vanishing (\cite[Corollary 1.3.3]{Nis}) yields 
\ml{thm:cycle-map0}{H^d(X_{\Nis}, V_{d, X|D})= E^{d,0}_2\\
= \Coker\left(\bigoplus_{x\in X_{(1)}} H^{d-1}_x(X_{\Nis},V_{d, X|D})\xr{\partial} 
\bigoplus_{x\in X_{(0)}} H^d_x(X_{\Nis}, V_{d, X|D})\right). }

\begin{claim}\label{thm:cycle-map-claim} 
Assume $C\subset X$ is an integral regular curve not contained in $D$ and $f\in k(C)^\times$ with 
$f\equiv 1$ mod $D_{|C}$. Then
\[\div(f)\in \bigoplus_{x\in U_{(0)}} \Z \subset \bigoplus_{x\in X_{(0)}} H^d_x(X_{\Nis}, V_{d, X|D}),\]
lies in the image of 
\eq{prop:cyle-map1}{\partial=(\partial_x)_x : k(C)^\times\cong H^{d-1}_c(X_{\Nis}, V_{d, X|D}) 
\to \bigoplus_{x\in X_{(0)}} H^d_x(X_{\Nis}, V_{d, X|D}),}
where $c$ is the generic point of $C$. 
\end{claim}
We prove the claim. 
We have $\partial(f)_{|U}= \pm \div_{C\cap U}(f)$ (e.g. \cite[(3.3.1)]{RS18}).
Thus it remains to show $\partial_x(f)=0$, for all $x\in D_{(0)}\cap C$.
To this end it suffices to show that $f$ lies in the image 
\eq{prop:cyle-map2}{H^{d-1}_{C_{(x)}}(\Spec \sO_{X,x}, V_{d, X|D}) \to H^{d-1}_c(X, V_{d, X|D}),}
for $x\in D_{(0)}\cap C$, where $C_{(x)}=\Spec \sO_{C,x}$. 
Under the isomorphism 
\[k(C)^\times \cong H^{d-1}_c(X, V_{d, X|D}) = H^{d-1}_c(X,K^M_{d,X})\]
$f$ corresponds  to (see \ref{para:coh-symb})
\eq{prop:cyle-map3}{\pm\genfrac{[}{]}{0pt}{}{\{\tilde{f}, s_1,\ldots, s_{d-1}\}}{s_1,\ldots, s_{d-1}},}
where $\tilde{f}\in \sO_{X,c}$ is a lift of $f$ and $s_1,\ldots, s_{d-1}\in \fm_c$ is a regular system of parameters.
In fact since $C$ and $X$ are  regular, the closed immersion  $C\inj X$ is  regular 
and hence we can choose the $s_i$'s
to be a regular sequence in $\sO_{X,x}$ generating the ideal sheaf of $C$ at $x$. Furthermore, since
$f\in \sO_{C|D_{|C},x}^\times$ by assumption,  we can choose the lift 
$\tilde{f}$ to lie in $\sO_{X|D,x}^\times$. 
Together with \ref{para:coh-symb}\ref{para:coh-symb2}, \ref{para:coh-symb3} 
this shows that \eqref{prop:cyle-map3} lies in the image of \eqref{prop:cyle-map2} 
 and proves Claim \ref{thm:cycle-map-claim}.

Assume $C\subset X$  is an arbitrary integral curve not contained in $D$. 
Let $\nu:\tC\to C$ be the normalization. We embed $\tC$ in $P_X=\P^n_X$ over $X$. 
Denote by $\pi:P_X\to X$ the projection and by $\pi_U$ its base change over $U$.
Consider the following diagram
\[\xymatrix{
Z_0(\tC\cap P_U)\ar[r]\ar[d]^{\nu_*} &
 H^{d+n}_{(P_{U})_{(0)}}(P_{U,\Nis}, K^M_{d+n, P_U})\ar[r]\ar[d]^{\tr_{\pi_{U}}}&
 H^{d+n}(P_{X,\Nis}, V_{d+n, P_X|P_D})\ar[d]^{\tr_\pi}\\
Z_0(C\cap U)\ar[r] & 
 H^{d}_{U_{(0)}}(U_{\Nis}, K^M_{d, U})\ar[r]&
H^d(X_\Nis, V_{d, X|D}),
}\]
where $\tr_{\pi}$ and $\tr_{\pi_U}$ are defined in \ref{defn:proj-tr} and  we view
$U_{(0)}$ and $(P_U)_{(0)}$ as families of closed supports on $U$ and $P_U$, respectively.
The left square commutes by \cite[Lemma 2.27]{RS18}, the right square by construction, and the
composite horizontal maps are equal to \eqref{para:CHM3}.
Since $\tC$ is regular Claim \ref{thm:cycle-map-claim} and \eqref{thm:cycle-map0} yield 
that $\div_{\tC}(f)$ is mapped to zero 
in $H^{d+n}(P_{X, \Nis}, V_{d+n, P_X|P_D})$, 
for $f\equiv 1$ mod $D_{|\tC}$. Hence
 $\div_C(f)=\nu_*\div_{\tC}(f)$ is  mapped to zero in $H^d(X_{\Nis}, V_{d, X|D})$.
This completes the proof.
\end{proof}

\begin{rmk}
The assumption that the support of $D$ is a simple normal crossing divisor, was only used
to apply the projective bundle formula, Theorem \ref{thm:proj-pf-relK}.
\end{rmk}

\providecommand{\bysame}{\leavevmode\hbox to3em{\hrulefill}\thinspace}
\providecommand{\MR}{\relax\ifhmode\unskip\space\fi MR }
\providecommand{\MRhref}[2]{%
  \href{http://www.ams.org/mathscinet-getitem?mr=#1}{#2}
}
\providecommand{\href}[2]{#2}

\end{document}